\documentclass[12pt]{amsart}
\usepackage{graphicx}
\usepackage[latin1]{inputenc}
\usepackage{amsmath,amssymb,amsfonts,mathrsfs}
\usepackage{amsthm,enumerate,multirow,multicol}

\usepackage{psfrag,color}
\usepackage[small,hang,bf,up,sf]{caption}

\newcounter{lletres}

\newtheorem{Talpha}[lletres]{Theorem}
\newtheorem{Lalpha}[lletres]{Lemma}

\newtheorem{thm}{Theorem}
\newtheorem{defi}[thm]{Definition}

\newtheorem{lma}[thm]{Lemma}

\newtheorem{obs}[thm]{Remark}

\newcommand{\D}{\mathbb D}

\newcommand{\N}{\mathbb N}

\newcommand{\U}{\mathcal U}
\newcommand{\B}{\mathcal B}

\newcommand{\Hi}{H^\infty}
\newcommand{\Hp}{H^p}
\newcommand{\eps}{\varepsilon}
\newcommand{\tdel}{\widetilde\Delta}
\newcommand{\hdel}{\widehat{\Delta}}

\begin{document}

\title[Interpolation and sampling for analytic selfmappings...]{Interpolation and sampling for analytic selfmappings of the disc.}
\author[Nacho Monreal Gal\'{a}n]{Nacho Monreal Gal\'{a}n}
\address{Nacho Monreal Gal\'{a}n, Departament of Mathematics, University of Crete, Voutes Campus, 70013 Heraklion, Crete, Greece.}
\email{nacho.mgalan@gmail.com}
\author[Michael Papadimitrakis]{Michael Papadimitrakis}
\address{Michael Papadimitrakis, Departament of Mathematics, University of Crete, Voutes Campus, 70013 Heraklion, Crete, Greece.}
\email{mihalis.papadimitrakis@gmail.com}

\thanks{
The first author is supported in part by the projects MTM2011-24606 and by the research project PE1(3378) implemented within the framework of the Action 
``Supporting Postdoctoral Researchers'' of the Operational Program ``Education and Lifelong Learning'' (Action's Beneficiary: General Secretariat for 
Research and Technology), co-financed by the European Social Fund (ESF) and the Greek State.}

\begin{abstract}
Two different problems are considered here. First, a version of Schwarz-Pick Lemma for $n$ points leads to an interpolation problem for analytic functions 
from the disc into itself, which may be considered as a particular case of the classical Nevanlinna-Pick interpolation problem. Second, a characterization 
of sampling sequences for this class of functions is given.
\end{abstract}

\maketitle

\section{Introduction.}
 
\noindent Let $\Hi$ be the space of bounded analytic functions in
the open unit disc $\D$ of the complex plane, and let $\U$
be the  set of functions $f \in \Hi$ with $\|f\|_\infty = \sup \{|f(z)| :
z \in \D\} \leq 1$. The Nevanlinna-Pick interpolation problem says:
\begin{equation}\label{NP}
\text{Given }\{z_n\},\,\{w_n\}\,\text{in }\D\text{, find }f\in\U\, :\,f(z_n)=w_n,\,n\in\N.
\end{equation}
Nevanlinna in \cite{Nevan} and Pick in \cite{Pick} proved independently that this problem has a solution
if and only if for all $N\in\N$ the matrix
$$\left(\frac{1-w_i\overline w_j}{1-z_i\overline z_j}\right)_{i,j=1,\dots,N}$$
is positive semidefinite. This theorem is the root of a very active field connected with many other topics, which may be found for example in \cite{AgMc}. 
However, the matrix condition is not easy to compute in general, and it does not give information about the geometry of the sequence $\{z_n\}$. Besides, the
Nevanlinna-Pick problem may be seen as a particular case of Carleson's celebrated result on interpolating sequences for $\Hi$, which we recall next.

\smallskip

\noindent A sequence of points $\{z_n\}$ in the unit disc is called an interpolating sequence if for every bounded sequence of values $\{w_n\}$ there exists a 
function $f\in\Hi$ such that $f(z_n)=w_n$, $n=1,2,\dots$  In his work \cite{Carleson} Carleson proved that $\{z_n\}$ is an interpolating sequence if and only if 
$\{z_n\}$ is a separated sequence and there exists a constant $M>0$ such that
$$\sum_{z_n\in Q}(1-|z_n|)\leq M\ell(Q)$$
for any Carleson square $Q$. A Carleson square is a set $Q$ of the form
\begin{equation*}
Q=\left\{r e^{i\theta}:\ 0<1-r<\ell(Q),\
|\theta-\theta_0|<\ell(Q)\right\}.
\end{equation*}
A sequence of points $\{z_n\}$ in the unit disc is called separated, with constant of separation $\eta>0$, if 
$$\displaystyle\inf_{i\neq
j}\beta(z_i,z_j)=\eta>0,$$ 
where $\beta(z,w)$ denotes the hyperbolic distance between $z$ and $w$ in $\D$.

\medskip

\noindent This geometric description of interpolating sequences has had a wide impact in Complex Analysis during the last decades, since it has offered a method 
of studying many different interpolation problems, as one may check for example in \cite{Seip}.

\medskip

\noindent In \cite{MMN} the authors considered a situation which may be understood as intermediate between the Nevanlinna-Pick and the Carleson interpolation problems. 
The setup of this specific problem was motivated by the Schwarz-Pick Lemma, which asserts that if $f\in\U$ and $z,w\in\D$, then 
$$\beta(f(z),f(w))\leq\beta(z,w).$$ 
Furthermore, equality holds for a pair (and then all) points of the disc if and only if $f$ is an automorphism of the disc (see for example \cite{Gar}). 

\medskip 

\noindent Our main result in this paper is a generalization of the result in \cite{MMN}. We will address a different interpolation problem, motivated by a suitable 
generalization of the Schwarz-Pick Lemma. This problem is described in Section 2, while Section 3 contains the proof of the main result.

\medskip

\noindent Besides, the Schwarz-Pick Lemma may motivate a definition of sampling sequence for the class of functions $\U$. Section 4 will be devoted to the study 
of this problem, so it may be read independently from the previous sections.

\bigskip

\section{Interpolating sequences of order $n$ for $\U$.}

\noindent Let us start by explaining in detail the result in \cite{MMN} that we will generalize here. In order to set an interpolation problem, the target space 
is restricted by the following definition: A sequence of points $Z=\{z_n\}$ in $\D$ is an interpolating sequence for $\U$ if there exists $\eps>0$ only depending 
on $Z$ such that for any sequence of values $W=\{w_n\}$ in $\D$ satisfying the compatibility condition
\begin{equation}\label{cc1}
\beta(w_m,w_n)\leq \eps \beta(z_m,z_n),\ \ \ n,m=1,2,\dots,
\end{equation}
then there exists $f\in\U$ such that $f(z_n)=w_n$ for $n\in\N$. With this definition the following characterization was proved.

\begin{Talpha}\label{th_MMN}
 A sequence $Z$ of distinct points in the unit disc is an interpolating sequence for $\U$ if and only if the following two conditions hold:
 \begin{itemize}
  \item[\textbf{(a)}] $Z=Z^{(1)}\cup Z^{(2)}$, where $Z^{(i)}$ is a separated sequence for $i=1,2$.
  \item[\textbf{(b)}] There exist constants $M>0$ and $0<\alpha<1$ such that for any Carleson square one has
   \begin{equation*}
    \#\left(Z\cap\left\{z\in Q\, :\, 2^{-m-1}\ell(Q)<1-|z|\leq2^{-m}\ell(Q)\right\}\right)\leq M2^{\alpha m}
   \end{equation*}
   for any $m=1,2,\dots$
 \end{itemize}
\end{Talpha}

\medskip

\noindent In \cite{MMN} the authors explain that the main condition in the description here is the density condition \textbf{(b)}, while the separation 
condition \textbf{(a)} appears because the problem is defined in terms of first differences. Then, one may wonder what may happen to this result if one modifies 
the definition using higher order differences. To this aim, we will need an appropriate generalization of the Schwarz-Pick Lemma, which is presented below.

\medskip

\noindent In \cite{BM} one may find a version of the classical Schwarz-Pick Lemma involving three points. This was extended in \cite{BRW} for $n$ points, doing a 
simple iteration of the result by Beardon and Minda. Both works pointed out the analogies between the role played by polynomials in the Euclidean setting and 
finite Blaschke products in the hyperbolic setting. Recall that a finite Blaschke product is a function of the form
$$\prod_{j=1}^n\frac{|z_j|}{z_j}\frac{z_j-z}{1-\overline z_j z},$$
where $z_j\in\D$. In order to state this result, Beardon and Minda first defined the complex pseudohyperbolic distance between $z,z_j\in\D$ as follows:
$$[z,z_j]:=\frac{z_j-z}{1-\overline z_j z},$$
which for a fixed $z_j$ represents an automorphism of the disc. This name is justified by the fact that the pseudohyperbolic 
distance between $z$ and $z_j$ is defined as 
$$\rho(z,z_j)=\left|[z,z_j]\right|.$$
Then the hyperbolic distance between $z$ and $z_j$ in $\D$ is 
$$\beta(z,z_j)=\log\frac{1+\rho(z,z_j)}{1-\rho(z,z_j)}.$$

\medskip

\noindent For a fixed $z_1\in\D$ and $f\in\U$ we define the hyperbolic difference quotient as
\begin{displaymath}
\Delta f(z;z_1):=\left\{\begin{array}{ccc}
\dfrac{[f(z),f(z_1)]}{[z,z_1]}\hfill&\text{ if }\hfill& z\in\D\setminus\{z_1\}\hfill\\\\
f^h(z_1)\hfill&\text{ if }\hfill& z=z_1,
\end{array}\right.
\end{displaymath}
where $f^h(z_1)$ is obtained as a limit and represents the hyperbolic derivative of $f$ at $z_1$, that is,
\begin{equation}\label{hyp_der}
 f^h(z_1)=\frac{(1-|z_1|^2)f'(z_1)}{1-|f(z_1)|^2}.
\end{equation}
The expression $\Delta f(z;z_1)$ defines a function in $\U$, since it is analytic and, as a consequence of the Swcharz-Pick Lemma, we have $|\Delta f(z;z_1)|\leq1$. 
We may now iterate this process to get differences of higher order. Writing $\Delta^0f(z)=f(z)$, we fix $z_1,\dots,z_n$ in $\D$ and for $k=1,\dots,n$ we 
define the $k$-th hyperbolic difference quotient as follows:
$$\Delta^k f(z;z_1,\dots,z_k)=\frac{[\Delta^{k-1}f(z;z_1,\dots,z_{k-1}),\Delta^{k-1}f(z_k;z_1,\dots,z_{k-1})]}{[z,z_k]},$$
interpreted as a limit when $z=z_k$. Clearly
$$\Delta^1(\Delta^{k-1}f(\cdot;z_1,\dots,z_{k-1}))(z;z_k)=\Delta^k f(z;z_1,\dots,z_k).$$
The following multi-point Schwarz-Pick Lemma appeared in \cite{BRW}.
\begin{Talpha}\label{SP_general}
 Fix pairwise distinct points $z_1,\dots,z_k$ in $\D$. Then, for all $f\in\U$ and $v,w$ in $\D$,
 $$\beta(\Delta^k f(v;z_1,\dots,z_k),\Delta^k f(w;z_1,\dots,z_k))\leq\beta(v,w).$$
 Equality holds for a pair of distinct points (and then for all) $v$ and $w$ in $\D$ if and only if $f$ is a Blaschke product of degree $k+1$.
\end{Talpha}

\medskip

\noindent Consider now two sequences $Z=\{z_n\}$ and $W=\{w_n\}$ in $\D$. In the line of \cite{MMN}, we want to define an interpolation problem in $\U$ and 
give conditions on $Z$ so that there exists a function in $\U$ interpolating $W$ at $Z$. To this end, let us define the hyperbolic difference quotients for the sequence $W$. 
Fix $z_1,\dots,z_n$ in $Z$ and the corresponding values $w_1,\dots,w_n$. Writing $\Delta^0_j:=w_j$, we define
$$\Delta^k_j:=\frac{[\Delta^{k-1}_j,\Delta^{k-1}_k]}{[z_j,z_k]},\text{ for }k=1,\dots,n-1\text { and }j=k+1,\dots,n.$$
It is easy to see that each $\Delta^k_j$ depends on $z_1,\dots,z_k,z_j$ and $w_1,\dots,w_k,w_j$. Moreover, if $f$ is a solution of (\ref{NP}) then
\begin{equation}\label{estab}
 \Delta^k_j=\Delta^kf(z_j;z_1,\dots,z_k)\,\,\text{ for }k=0,\dots,n-1,\,\,j=k+1,\dots,n,
\end{equation}
\noindent (see \cite[Lemma 4.2]{BRW}) so necessarily $|\Delta^k_j|\leq1$ by Theorem \ref{SP_general}.

\medskip

\noindent The complete list of hyperbolic difference quotients may be represented in the following table, which will be very useful for our aim:
\begin{displaymath}
 \begin{array}{cc|llccc|}
  \cline{3-7} 
  & & \multicolumn{5}{c|}{\text{hyperbolic difference quotients}}\\
  \hline
  \multicolumn{1}{|c|}{Z} & \multicolumn{1}{|c|}{W} & 1 & 2 & \cdots & n-2 & n-1 \\
  \hline 
  \multicolumn{1}{|c|}{\vspace{0cm}} & \multicolumn{1}{|c|}{\vspace{0cm}} & & & & & \\
  \multicolumn{1}{|c|}{z_1} & \multicolumn{1}{|c|}{w_1=\Delta^0_1} & & & & & \\ 
  \multicolumn{1}{|c|}{\vspace{0cm}} & \multicolumn{1}{|c|}{\vspace{0cm}} & & & & & \\
  \multicolumn{1}{|c|}{z_2} & \multicolumn{1}{|c|}{w_2=\Delta^0_2} & \Delta^1_2 & & & & \\
  \multicolumn{1}{|c|}{\vspace{0cm}} & \multicolumn{1}{|c|}{\vspace{0cm}} & & & & & \\
  \multicolumn{1}{|c|}{z_3} & \multicolumn{1}{|c|}{w_3=\Delta^0_3} & \Delta^1_3 & \Delta^2_3 & & &\\
  \multicolumn{1}{|c|}{\vspace{0cm}} & \multicolumn{1}{|c|}{\vspace{0cm}} & & & & & \\
  \multicolumn{1}{|c|}{\vdots} & \multicolumn{1}{|c|}{\vdots} & \multicolumn{1}{c}{\vdots} & \multicolumn{1}{c}{\vdots} & \ddots & &\\
  \multicolumn{1}{|c|}{\vspace{0cm}} & \multicolumn{1}{|c|}{\vspace{0cm}} & & & & & \\
  \multicolumn{1}{|c|}{z_{n-1}} & \multicolumn{1}{|c|}{w_{n-1}=\Delta^0_{n-1}} & \Delta^1_{n-1} & \Delta^2_{n-1} & \cdots & \Delta^{n-2}_{n-1} & \\
  \multicolumn{1}{|c|}{\vspace{0cm}} & \multicolumn{1}{|c|}{\vspace{0cm}} & & & & & \\
  \multicolumn{1}{|c|}{z_{n}} & \multicolumn{1}{|c|}{w_{n}=\Delta^0_{n-1}} & \Delta^1_{n} & \Delta^2_{n} & \cdots & \Delta^{n-2}_{n} & \Delta^{n-1}_n\\
  \multicolumn{1}{|c|}{\vspace{0cm}} & \multicolumn{1}{|c|}{\vspace{0cm}} & & & & & \\
  \hline
 \end{array}
\end{displaymath}

\noindent Observe that the definition of interpolating sequence given in the previous section may be rewritten as 
$$\beta(\Delta^0_i,\Delta^0_j)\leq \eps\beta(z_i,z_j)$$
for any $i,j\in\N$. Also, we should remark that the hyperbolic difference quotients depend on the order given to the points $z_1,\dots,z_n$. The triangle 
formed in the table will be called triangle of hyperbolic difference quotients. 

\medskip

\noindent Fix a sequence of nodes $Z$ and $\eps>0$. We will say that a sequence $W$ satisfies the 
$\eps$-compatibility condition for $Z$ if for any $\{z_1,\dots,z_n\}\subset Z$ the terms of the corresponding triangle of hyperbolic difference quotients satisfy the inequality
\begin{equation}\label{ecc}
 \beta(\Delta^{k}_i,\Delta^{k}_j)\leq\eps\beta(z_i,z_j)\,\,\text{ for }k=0,\dots,n-2,\,\,i,j=k+1,\dots n.
\end{equation}
 
\noindent The constant $\eps$ will be called interpolation constant. Now we can give a definition of interpolating sequence, based on the triangle of hyperbolic 
difference quotients, analogous to the one given in \cite{MMN}.

\begin{defi}\label{def2}
 A sequence $Z$ of points in $\D$ is an interpolating sequence of order $n-1$ for $\U$ if there exists $\eps>0$, depending only on the sequence and $n$, such 
 that for any sequence $W\subset\D$ satisfying the $\eps$-compatibility condition \textit{(\ref{ecc})} there exists $f\in\U$ such that $f(z_j)=w_j$ for 
 $j\in\N$.
\end{defi}

\noindent As in \cite{MMN} this definition is conformally invariant, in the sense that if $Z$ is an interpolating sequence of order $n-1$ and $\tau$ is an 
automorphism of the disc then the sequence $\tau(Z)=\{\tau(z_j)\}$ is also an interpolating sequence of order $n-1$ with the same interpolation constant $\eps$. 
This is a consequence of the following property that may be found in \cite[Lemma 3.3]{Riv}: Let $f\in\U$, let $\phi$ and $\psi$ be two automorphisms of the disc 
and take pairwise distinct points $z_1,\dots,z_n$ in $\D$. Then for $1\leq j\leq n$ there exists $\theta\in [0,2\pi)$ such that 
$$\Delta^j(\psi\circ f\circ\phi)(z;z_1,\dots,z_j)=e^{i\theta}\Delta^j(f(\phi(z));\phi(z_1),\dots,\phi(z_j))\,\, z\in\D.$$ 
\noindent Besides, the definition coincides with the one given in \cite{MMN} when $n=2$. It is important to remark a direct consequence of the 
definition: if $Z$ is an interpolating sequence of order $k$ for $\U$, then it is also an interpolating sequence of order $j$ for $\U$ with $j=1,\dots,k-1$.

\medskip

\noindent The main result of this work is the characterization of these interpolating sequences.

\begin{thm}\label{main_th}
 A sequence $Z$ of distinct points in the unit disc is an interpolating sequence of order $n-1$ for $\U$ if and only if the following two conditions hold:
 \begin{itemize}
  \item[\textbf{(a)}] $Z=Z^{(1)}\cup\dots\cup Z^{(n)}$, where $Z^{(i)}$ is a separated sequence for $i=1,\dots,n$.
  \item[\textbf{(b)}] There exist constants $M>0$ and $0<\alpha<1$ such that for any Carleson square one has
   \begin{equation*}
    \#\left(Z\cap\left\{z\in Q\, :\, 2^{-m-1}\ell(Q)<1-|z|\leq2^{-m}\ell(Q)\right\}\right)\leq M2^{\alpha m}
   \end{equation*}
   for any $m=1,2,\dots$
 \end{itemize}
\end{thm}

\noindent Observe that condition \textbf{(b)} in Theorem \ref{main_th}, that is, the density condition, remains the same as in Theorem \ref{th_MMN}, hence the 
only change is on the separation condition \textbf{(a)}, as it was expected. Furthermore, other previous results on generalizations of interpolation problems 
for Hardy spaces pointed out to this fact. In \cite{Vas2} and \cite{Vas1} the author generalizes the Carleson Theorem on interpolating sequences for $\Hi$. A 
suitable notion of hyperbolic difference quotient is given there. Then imposing the boundedness of the difference quotient up to a certain order leads to changing the separation condition in Carleson's 
result. Also in \cite{BNO} the authors generalize the interpolation problem for $\Hp$ solved previously in \cite{SS} (case $p\geq1$) and \cite{Kab} (case $0<p<1$). 
In that work, the trace space is defined by means of a certain maximal function using those hyperbolic difference quotients, and the separation condition changes in the same direction as in our result. Also these problems are studied in \cite{Har1} and \cite{Har2}.

\medskip

\noindent The proof of the necessity of the density condition \textbf{(b)} may be seen as a consequence of Theorem \ref{th_MMN}, and as we said above, the separation condition \textbf{(a)} will come from the definition of the trace space. For the sufficiency, we will split the sequence $Z$ in $n$ separated sequences $Z^{(1)},\dots,Z^{(n)}$, and 
use the solution of the interpolation problem constructed in \cite{MMN} for $Z^{(1)}$. This solution will have some extra properties that will let us define an 
auxiliary interpolation problem for $Z\setminus Z^{(1)}$. An inductive argument will lead us to the solution for the whole sequence $Z$.

\bigskip

\noindent Before starting with the proof, let us discuss the relation of our result here with the classical Nevanlinna-Pick problem. Theorem \ref{main_th}, as well as the one in \cite{MMN}, may be considered as a particular case of a general Nevanlinna-Pick interpolation problem. Actually, in \cite{BRW} the following version of the Nevanlinna-Pick Theorem was proved.

\begin{Talpha}\label{NP_BRW}
 Fix pairwise distinct interpolation points $z_1,\dots,z_n\in\D$ and the corresponding interpolation values $w_1,\dots,w_n\in\D$. Then problem (\ref{NP}) has 
 infinitely many solutions if and only if one (and then all) of the following conditions hold:
 \begin{enumerate}[\bf (i)]
  \item $|\Delta^{n-1}_n|<1$.
  \item $|\Delta^{k}_{k+1}|<1$ if $1\leq k\leq n-1$.
  \item $|\Delta^{k}_j|<1$ if $1\leq k<j\leq n$.
  \item $\beta(\Delta^{k-1}_j,\Delta^{k-1}_k)<\beta(z_j,z_k)$ if $1\leq k<j\leq n$.
  \item $\beta(\Delta^{k-1}_j,\Delta^{k-1}_i)<\beta(z_j,z_i)$ if $1\leq k<i<j\leq n$.
 \end{enumerate}
\end{Talpha}

\noindent Comparing this result with Definition \ref{def2}, in our case we are imposing a more restrictive condition on the hyperbolic distance between any two 
hyperbolic difference quotients. In fact, the equivalence of \textbf{(iii)} and \textbf{(iv)}, which is a direct consequence of the monotonic dependence of 
$\beta$ on $\rho$, does not hold in general if we replace the natural Lipschitz constant 1 there by $\eps$. Nevertheless, it holds when the points 
$z_1,\dots,z_n$ are close. And this is the important case in the proof of Theorem \ref{main_th}.

\medskip

\noindent Finally we should also remark that a proof of \textbf{(ii)} in Theorem \ref{NP_BRW} may be essentially found also in \cite[Chapter X]{Wal}. Moreover, this book refers to a paper by Denjoy \cite{Den} where it was proved a condition for the infinite points case using the hyperbolic difference quotients: The problem (\ref{NP}) has infinitely many solutions if and only if
$$\sum_{n=1}^\infty\frac{1-|z_n|}{1-|\Delta^{n-1}_n|}<\infty.$$
 
\bigskip

\section{Proof of Theorem \ref{main_th}.}

\noindent We start by stating two results that will be basic in the proof. The first one is the following Lemma, which is a simple consequence of the 
Classical Schwarz-Pick Lemma. We will denote $D_h(z,\eta)$ the hyperbolic disc with center $z$ and radius $\eta$.

\begin{lma}\label{lemma1}
 Let $f$ be an analytic function in $D_h(z_0,\eta_1)\subseteq\D$ such that $|f(z)|\leq C$ on $D_h(z_0,\eta_1)$. Let $a\in D_h(z_0,\eta)$ with fixed $0<\eta<\eta_1$. 
 Then
 $$|f(z)-f(a)|\leq \widetilde{C}\rho(z,a),\,\,\text{ for }z\in D_h(z_0,\eta_1),$$
where $\widetilde{C}>0$ depends on $\eta_1-\eta$ and $C$.
\end{lma}

\medskip

\noindent The second one is explained in the following remark.

\begin{obs}\label{remark}
 Let $Z=\{z_1,\dots,z_n\}$ be a sequence of distinct points in a small hyperbolic disc of radius $\eta>0$, and let $W=\{\Delta^0_1,\dots,\Delta^0_n\}$ be a set of values. Consider the corresponding hyperbolic difference quotients. If for a certain $\eps>0$ one has
 \begin{equation}\label{column_cond}
  \beta(\Delta^k_j,\Delta^k_{k+1})\leq\eps\beta(z_j,z_{k+1})
 \end{equation}
 \noindent for $1\leq k\leq n-2$ and $k+2\leq j\leq n-1$, then there exists $C>0$ depending only on $\rho$ such that 
 $$\beta(\Delta^k_i,\Delta^k_j)\leq C\eps\beta(z_i,z_j)$$
 for $k=1,\dots,n-2$ and $k+2\leq i<j\leq n-1$.
\end{obs}

\begin{proof}
 We will follow the ideas of the proof of Theorem \ref{NP_BRW}. Consider the following triangle of hyperbolic difference quotients (we will refer to the last 
 row later):
 \begin{displaymath}
 \begin{array}{cc|lllllc|}
  \cline{3-8} 
  & & \multicolumn{6}{c|}{\text{hyperbolic difference quotients}}\\
  \hline
  \multicolumn{1}{|c|}{Z} & \multicolumn{1}{|c|}{W} & 1 & 2 & \cdots & n-2 & n-1 & n \\
  \hline 
  \multicolumn{1}{|c|}{\vspace{0cm}} & \multicolumn{1}{|c|}{\vspace{0cm}} & & & & & & \\
  \multicolumn{1}{|c|}{z_1} & \multicolumn{1}{|c|}{\Delta^0_1} & & & & & & \\ 
  \multicolumn{1}{|c|}{\vspace{0cm}} & \multicolumn{1}{|c|}{\vspace{0cm}} & & & & & & \\
  \multicolumn{1}{|c|}{z_2} & \multicolumn{1}{|c|}{\Delta^0_2} & \Delta^1_2 & & & & & \\
  \multicolumn{1}{|c|}{\vspace{0cm}} & \multicolumn{1}{|c|}{\vspace{0cm}} & & & & & & \\
  \multicolumn{1}{|c|}{z_2} & \multicolumn{1}{|c|}{\Delta^0_3} & \Delta^1_3 & \Delta^2_3 & & & & \\
  \multicolumn{1}{|c|}{\vspace{0cm}} & \multicolumn{1}{|c|}{\vspace{0cm}} & & & & & & \\
  \multicolumn{1}{|c|}{\vdots} & \multicolumn{1}{|c|}{\vdots} & \multicolumn{1}{c}{\vdots} & \multicolumn{1}{c}{\vdots} & \ddots & & & \\
  \multicolumn{1}{|c|}{\vspace{0cm}} & \multicolumn{1}{|c|}{\vspace{0cm}} & & & & & & \\
  \multicolumn{1}{|c|}{z_{n-1}} & \multicolumn{1}{|c|}{\Delta^0_{n-1}} & \Delta^1_{n-1} & \Delta^2_{n-1} & \cdots & \Delta^{n-2}_{n-1} & & \\
  \multicolumn{1}{|c|}{\vspace{0cm}} & \multicolumn{1}{|c|}{\vspace{0cm}} & & & & & & \\
  \multicolumn{1}{|c|}{z_{n}} & \multicolumn{1}{|c|}{\Delta^0_n} & \Delta^1_{n} & \Delta^2_{n} & \cdots & \Delta^{n-2}_{n} & \Delta^{n-1}_{n} & \\
  \multicolumn{1}{|c|}{\vspace{0cm}} & \multicolumn{1}{|c|}{\vspace{0cm}} & & & & & & \\
  \multicolumn{1}{|c|}{z} & \multicolumn{1}{|c|}{g(z)} & \Delta^1g(z) & \Delta^2g(z) & \cdots & \Delta^{n-2}g(z) & \Delta^{n-1}g(z) & \Delta^{n}g(z)\\
  \multicolumn{1}{|c|}{\vspace{0cm}} & \multicolumn{1}{|c|}{\vspace{0cm}} & & & & & & \\
  \hline
 \end{array}
\end{displaymath}

\medskip

\noindent Since the points are in a small hyperbolic disc $D_h$, we can consider the pseudohyperbolic distance $\rho$ instead of $\beta$. Hence 
(\ref{column_cond}) is equivalent to
$$|\Delta^k_j|\leq\eps\,\,\text{ for }k=1,\dots,n-1\,\text{ and }j=k+1,\dots,n.$$
\noindent In particular, all of them are smaller than 1, so by Theorem \ref{NP_BRW} there are infinitely many solutions of the interpolation problem 
(\ref{NP}) for the sequence of $n$ points $Z$. Actually, in the proof appearing in \cite{BRW} we can find the following version of the Schur's algorithm to construct these 
solutions (see the cited paper for the details). The last row of the table, where $$\Delta^kg(z):=\Delta^kg(z;z_1,\dots,z_k)$$ 
for $k=1,\dots,n$, has been added in order to describe this process. Take an arbitrary function $g_0(z)\in\U$ and apply the recursive formula
\begin{equation}\label{formula_NP}
 g_k(z)=\left[[z,z_{n+1-k}]\cdot g_{k-1}(z),\Delta^{n-k}_{n+1-k}\right],\,\,\text{for }k=1,\dots,n
\end{equation}
\noindent to obtain a solution of the problem for $Z$. Denoting $g:=g_n\in\U$, we easily see that this is a solution, since $g_n(z_j)=\Delta^0_j$ for 
$j=1,\dots n$. Moreover, the functions of the algorithm satisfy 
$$g_{n-k}(z)=\Delta^kg(z)\,\,\text{ for }k=0,\dots,n.$$
\noindent In particular if we choose $g_0(z)=0$ then it is easy to see that there exists $C>0$ depending only on $n$ such that $|g_k(z)|\leq C\eps$ for 
$k=1,\dots,n-1$. Now property (\ref{estab}) and Lemma \ref{lemma1} applied in the whole disc $\D$ show that for  $1\leq k\leq n-2$ 
$$|\Delta^k_j-\Delta^k_i|=|g_{n-k}(z_j)-g_{n-k}(z_i)|\leq C\eps\rho(z_j,z_i)$$
for $k+1\leq i<j\leq n$. This finishes the proof.
\end{proof}

\noindent Observe that Remark \ref{remark} is stated for a fixed ordering of the $n$ points and $k\geq 1$. The case $k=0$ holds if we suppose that 
the $\eps$-compatibility condition is true for any permutation of the points. Hence, in order to check inequality (\ref{ecc}) for $k=0,\dots,n-2$ when the 
$n$ points are in a small hyperbolic disc we just need to see that 
\begin{equation}\label{ecc2}
 |\Delta^k_j|\leq\eps\,\,\text{ for }k=1,\dots,n-1\,\text{ and }j=k+1,\dots,n,
\end{equation}
\noindent and we need to check it for any permutation of the points.

\medskip
 
\subsection{Necessity}

\noindent Let $Z\subset\D$ be an interpolating sequence of order $n-1$ for $\U$. We need to show that there exists $\eta>0$ such that there cannot be more 
than $n$ points inside a small hyperbolic disc $D_h$ of radius $\eta$. To this end, take $z_1,z_2,\dots,z_n,z_{n+1}\in Z$ such that $z_j\in\D_h(z_1,\eta)$ 
for $j=2,\dots,n+1$. Suppose also that 
$$\min\rho(z_1,z_j)=\rho(z_1,z_{n+1}),\,\,\text{ where }j=2,\dots,n+1.$$
\noindent Now we can define the following values: $w_j=0$ for $j=1,\dots,n$, and $w_{n+1}=\eps x$, where $x$ will be determined so that the 
sequence $W=\{w_1,\dots,w_{n+1}\}$ satisfies the $\eps$-compatibility condition. To this end, we have to take the set $Z^j=\{z_1,\dots,z_{n+1}\}\setminus\{z_j\}$, 
for $j=1,\dots,n+1$ and check inequality (\ref{ecc2}) for the terms in the corresponding triangle of hyperbolic difference quotients for any permutation of 
the points of $Z^j$. We will see that it is enough to check it only for a certain permutation, with a suitable choice of the value $x$.

\medskip

\noindent The inequality trivially holds for the set $Z^{n+1}$, since all the hyperbolic differences vanish. For the sets $Z^j$ with $j=1,\dots,n$, we will first 
consider the original ordination of the points, and then we will consider any permutation of them. For the sake of simplicity, we will do it only in the case of 
$Z^1$, but the argument is the same for all other sets. Consider the following triangle of hyperbolic difference quotients.

\begin{displaymath}
 \begin{array}{cc|lllll|}
  \cline{3-7} 
  & & \multicolumn{5}{c|}{\text{hyperbolic difference quotients}}\\
  \hline
  \multicolumn{1}{|c|}{Z} & \multicolumn{1}{|c|}{W} & 1 & 2 & \cdots & n-2 & n-1 \\
  \hline 
  \multicolumn{1}{|c|}{\vspace{0cm}} & \multicolumn{1}{|c|}{\vspace{0cm}} & & & & & \\
  \multicolumn{1}{|c|}{z_2} & \multicolumn{1}{|c|}{0} & & & & & \\ 
  \multicolumn{1}{|c|}{\vspace{0cm}} & \multicolumn{1}{|c|}{\vspace{0cm}} & & & & & \\
  \multicolumn{1}{|c|}{z_3} & \multicolumn{1}{|c|}{0} & \Delta^1_3 & & & & \\
  \multicolumn{1}{|c|}{\vspace{0cm}} & \multicolumn{1}{|c|}{\vspace{0cm}} & & & & & \\
  \multicolumn{1}{|c|}{z_4} & \multicolumn{1}{|c|}{0} & \Delta^1_4 & \Delta^2_4 & & & \\
  \multicolumn{1}{|c|}{\vspace{0cm}} & \multicolumn{1}{|c|}{\vspace{0cm}} & & & & & \\
  \multicolumn{1}{|c|}{\vdots} & \multicolumn{1}{|c|}{\vdots} & \multicolumn{1}{c}{\vdots} & \multicolumn{1}{c}{\vdots} & \ddots & & \\
  \multicolumn{1}{|c|}{\vspace{0cm}} & \multicolumn{1}{|c|}{\vspace{0cm}} & & & & & \\
  \multicolumn{1}{|c|}{z_{n}} & \multicolumn{1}{|c|}{0} & \Delta^1_{n} & \Delta^2_{n} & \cdots & \Delta^{n-2}_{n} & \\
  \multicolumn{1}{|c|}{\vspace{0cm}} & \multicolumn{1}{|c|}{\vspace{0cm}} & & & & & \\
  \multicolumn{1}{|c|}{z_{n+1}} & \multicolumn{1}{|c|}{\eps x} & \Delta^1_{n+1} & \Delta^2_{n+1} & \cdots & \Delta^{n-2}_{n+1} & \Delta^{n-1}_{n+1} \\
  \multicolumn{1}{|c|}{\vspace{0cm}} & \multicolumn{1}{|c|}{\vspace{0cm}} & & & & & \\
  \hline
 \end{array}
\end{displaymath}

\medskip

\noindent It is easy to see that the only hyperbolic difference quotients that do not vanish in general are the ones in the row corresponding to the point 
$z_{n+1}$, that is, the ones of the form $\Delta^k_{n+1}$ for $k=1,\dots,n-1$, and moreover,
$$\Delta^k_{n+1}=\frac{\Delta^{k-1}_{n+1}}{[z_{n+1},z_{k+1}]},$$
\noindent where $\Delta^0_{n+1}=\eps x$. Consequently, (\ref{ecc2}) is true in this case if
$$|x|\leq\prod_{i=2}^{n}\rho(z_i,z_{n+1}),$$
Without loss of generality, we may suppose that
$$\max\rho(z_j,z_{n+1})=\rho(z_n,z_{n+1}),\,\,\text{ for }j=1,\dots,n.$$
Then, we can define 
$$x=C\cdot\displaystyle\prod_{j=1}^{n-1}[z_{n+1},z_j],$$ 
with $0<C<1$. Hence, $|\Delta^k_j|\leq C\eps$ for $k=1,\dots,n-1$ and $j=k+2,\dots,n+1$.

\medskip

\noindent Next, we are going to see that for a suitable choice of $C$ inequality (\ref{ecc2}) holds for the terms of the triangle corresponding to any 
permutation of the points. To this end, we will proceed as in the proof of Remark \ref{remark}. Let $g_0(z)=0$, since all the terms in the upper diagonal but the last one vanish, the recursive formula (\ref{formula_NP}) produces the following solution of the interpolation problem for $Z^{1}$:
$$g(z)=(-1)^{n+1}\Delta^{n-1}_{n+1}\prod_{j=2}^n[z,z_{j}],$$
which easily satisfies $|g(z)|\leq\eps$. Hence, choosing $C$ small enough, Lemma \ref{lemma1} and property (\ref{estab}) show that for any permutation 
of the points inequality (\ref{ecc2}) holds. The constant $C$ does not depend on the points, so the choice of $x$ is the same for any subsequence $Z^j$, 
for $j=1,\dots,n$. 

\medskip 

\noindent Since $Z$ is an interpolating sequence of order $n-1$ there is a function $f\in\U$ such that $f(z_j)=w_j$ for $j=1,\dots n+1$. Besides, the function 
has the form
$$f(z)=\prod_{j=1}^{n}[z,z_j]h(z),$$
where $h\in\U$. Now evaluating the function at $z_{n+1}$ we see that
$$\eps\prod_{j=1}^{n-1}\rho(z_j,z_{n+1})=|f(z_{n+1})|\leq\prod_{j=1}^n\rho(z_j,z_{n+1}).$$
Hence, there exists $C>0$ such that $\eta>C\eps$, which tells that there cannot be more than $n$ points inside a small hyperbolic disc. So \textbf{(a)} holds.

\bigskip 

\noindent The necessity of condition \textbf{{(b)}} is a consequence of Theorem \ref{th_MMN}, since in particular each separated sequence $Z^{(j)}$ is an 
interpolating sequence of order 1.

\bigskip

\subsection{Sufficiency.}

\noindent Consider a sequence $Z$  of points in $\D$ satisfying conditions \textbf{(a)} and \textbf{(b)} of Theorem \ref{main_th}. Assume that there exist 
$\eps>0$ (that will be fixed later) and also a sequence of values $W\subset\D$ which satisfies the $\eps$-compatibility condition. We 
take first the separated sequence $Z^{(1)}$ and its corresponding values $W^{(1)}$. Condition \textbf{(b)} and the $\eps$-compatibility condition 
for first differences allows us to construct a function $f_1\in\U$ that interpolates $W^{(1)}$ at $Z^{(1)}$. Furthermore, this function has two additional properties. In 
order to state them, let $E_1(z)$ be the outer function with boundary values $1-|f_1(e^{i\theta})|$, that is,
$$E_1(z)=\exp\left(\frac{1}{2\pi}\int_0^{2\pi}\frac{e^{i\theta}+z}{e^{i\theta}-z}\log{\left(1-|f_1(e^{i\theta})|\right)}d\theta\right).$$
\noindent The construction of $f_1$ is explained in detail in \cite[Section 4]{MMN} and its properties are stated in the following Lemma. 

\begin{Lalpha}
 Let $Z^{(1)}=\{z^{(1)}_i\}$ be a separated sequence in $\D$ with interpolation constant $\eps>0$. Let $W^{(1)}=\{w^{(1)}_i\}$ be a sequence in $\D$ such that 
 the compatibility condition (\ref{cc1}) holds. Then there exists $f_1\in\U$ such that $f_1(z^{(1)}_i)=w^{(1)}_i$ for $i\in\N$. Moreover, there exists a 
 constant $C>0$ such that 
 \begin{equation}\label{prop1}
  |E_1(z)|\geq C(1-|f_1(z)|).
 \end{equation}
 for every $z\in\D$. Besides, there exists $\eta_1>0$ depending on $Z^{(1)}$ such that for $z\in D_h(z^{(1)}_i,\eta_1)$ one has that
 \begin{equation}\label{prop2}
  \beta(f_1(z),f_1(z^{(1)}_i))\leq C\eps\beta(z,z^{(1)}_i).
 \end{equation}
\end{Lalpha}

\medskip

\noindent Let $Z^{(j)}=\{z^{(j)}_i\}$ for $j=2,\dots,n$. As in \cite{MMN}, we can take $\eta>0$ to be smaller than the separation constant of the 
sequence $Z^{(1)}$ and also assume that $\eta<\eta_1$, where $\eta_1$ is the radius appearing in (\ref{prop2}). Then we can suppose that 
$$Z^{(j)}\subset\bigcup_{i\in\N} D_h(z^{(1)}_i,\eta)\,\,\text{ for every }j=2,\dots,n.$$
This means that for each $z^{(j)}_m\in Z^{(j)}$ there exists $z^{(1)}_{i(m)}\in Z^{(1)}$ such that $z^{(j)}_m\in D_h(z^{(1)}_{i(m)},\eta)$. Now we will state 
and solve an auxiliary interpolation problem for $Z\setminus Z^{(1)}$ as follows. 

\medskip

\noindent Let $B_1(z)$ be the Blaschke product with zeros $Z^{(1)}$, and for $j=2,\dots,n$ let
$$\widetilde{w}^{(j)}_m=\frac{f_1(z^{(j)}_m)-w^{(j)}_m}{B_1(z^{(j)}_m)\cdot E_1(z^{(j)}_m)}.$$

\noindent These auxiliary values were also defined in \cite{MMN}. Hence we can address the interpolation problem (\ref{NP}) for 
$Z\setminus Z^{(1)}$ and $\widetilde{W}=\cup\widetilde{W}^{(j)}$, with $\widetilde{W}^{(j)}=\{\widetilde{w}^{(j)}_m\}$ for each $j=2,\dots,n$. If there exists 
a solution $\widetilde{f}\in\U$ such that $\widetilde{f}(z^{(j)}_m)=w^{(j)}_m$ for $j=2,\dots,n$ and $m\in\N$ then the function
\begin{equation}\label{solution}
 f(z)=f_1(z)-B_1(z)\cdot E_1(z)\cdot \widetilde{f}(z)
\end{equation}
\noindent is in $\U$ and interpolates the values $W$ in the nodes $Z$. 

\bigskip

\noindent We will proceed by induction in $n$. The case $n=2$ is actually the proof of the sufficiency of Theorem \ref{th_MMN}. So let us suppose that Theorem \ref{main_th} holds for $n-1$ and prove it 
for $n$.

\medskip

\noindent In order to make the proof easier, we will split the auxiliary problem into two different interpolation problems, since

$$\widetilde{w}^{(j)}_m=\frac{f_1(z^{(j)}_m)-f_1(z^{(1)}_{i(m)})}{B_1(z^{(j)}_m)\cdot E_1(z^{(j)}_m)}+\frac{w^{(1)}_{i(m)}-w^{(j)}_m}{B_1(z^{(j)}_m)\cdot E_1(z^{(j)}_m)}.$$

\medskip

\noindent Consider first the interpolation problem defined by the values

\begin{equation*}
 \hdel^{(j)}_m=2\cdot\frac{f_1(z^{(j)}_m)-f_1(z^{(1)}_{i(m)})}{B_1(z^{(j)}_m)\cdot E_1(z^{(j)}_m)},
\end{equation*}

\noindent corresponding to $z^{(j)}_m$ for $j=2,\dots,n$ and $m\in\N$. Observe that properties (\ref{prop1}) and (\ref{prop2}) imply that 
$|\hdel^{(j)}_m|\lesssim\eps$. Now we want to apply the hypothesis of induction to get the existence of the function $h\in\U$ such that $h(z^{(j)}_m)=\hdel^{(j)}_m$ 
for $j=2,\dots,n$ and $m\in\N$. In order to do this, we will choose $n-1$ points of the sequence $Z\setminus Z^{(1)}$ and check the $\eps$-compatibility condition 
for the corresponding triangle of hyperbolic difference quotients. 

\noindent To this aim we have basically two different cases: when the points are far and when they are close, that is, when the distance between the points is bigger than a fixed constant, or when it is smaller. Let us start by the first case, and in order to simplify the notation, take first $z_2,\dots,z_n\in Z\setminus Z^{(1)}$ such that $\beta(z_i,z_j)\geq\eta$ for $2\leq i<j\leq n$. Denote the corresponding images by 
$\hdel^0_j$ for $j=2,\dots,n$. Then the triangle of hyperbolic difference quotients is 
\begin{displaymath}
 \begin{array}{|c|c|cccc|}
  \hline
  Z & W & 1 & 2 & \cdots & n-2 \\
  \hline
  \vspace{0cm} & & & & & \\
  z_2 & \hdel^0_2 & & & & \\
  \vspace{0cm} & & & & & \\
  z_3 & \hdel^0_3 & \hdel^1_3 & & & \\
  \vspace{0cm} & & & & & \\
  z_4 & \hdel^0_4 & \hdel^1_4 & \hdel^2_4 & & \\
  \vspace{0cm} & & & & & \\
  \vdots & \vdots & \vdots & \vdots & \ddots & \\
  \vspace{0cm} & & & & & \\
  z_n & \hdel^0_n & \hdel^1_n & \hdel^2_n & \cdots & \hdel^{n-2}_n \\
  \vspace{0cm} & & & & & \\
  \hline
 \end{array}
\end{displaymath}

\noindent where
$$\hdel^k_j=\frac{\left[\hdel^{k-1}_j,\hdel^{k-1}_{k+1}\right]}{[z_j,z_{k+1}]}\,\,\text{ for }k=1,\dots,n-2\,\text{ and }j=k+2,\dots,n.$$

\noindent Since $|\hdel^0_j|\lesssim\eps$, if $\eps$ is small enough one may see that $\beta(\hdel^k_j,\hdel^k_i)\leq C\eps$, so then there exists 
$C'=C'(\eta)>0$ such that 
$$\beta(\hdel^k_j,\hdel^k_i)\leq C'\eps\beta(z_j,z_i)\,\,\text{ for }0\leq k\leq n-2,\text{ and }k+2\leq i<j\leq n.$$

\medskip 

\noindent Hence we may focus on the case of close points. We can consider the following situation: Take $z_1\in Z^{(1)}$ so that for each $j=2,\dots,n$ there 
exists a unique $z_j\in Z^{(j)}\cap D_h(z_1,\eta)$. Denote $\Omega_1=D_h(z_1,\eta_1)$, so then $\rho(z_j,\partial\Omega_1)\geq \eta_1-\eta>0$, that is, $z_j$ 
are further than a fixed distance from $\partial\Omega_1$. Then the corresponding images are
$$\hdel^0_j=2\cdot\frac{f_1(z_j)-f_1(z_1)}{B_1(z_j)E_1(z_j)}\,\,\text{ for}j=2,\dots,n.$$
\noindent We want to see that $|\hdel^k_j|\leq C \eps$, with $C$ depending only on the sequence $\{z_2,\dots,z_n\}$. Consider the function
$$f_2(z)=2\frac{f_1(z)-f_1(z_1)}{B_1(z)E_1(z)},$$
which by (\ref{prop1}) and (\ref{prop2}) is analytic and $|f_2(z)|\lesssim\eps$ in $\Omega_1$. Moreover, $f_2(z_j)=\hdel^0_j$ for $j=2,\dots,n$. Observe that 
Lemma \ref{lemma1} implies that $|f_2(z_2)-f_2(z)|\lesssim\eps\rho(z,z_2)$ for $z\in\Omega_1$. Now, the function
$$f_3(z):=\frac{[f_2(z),f_2(z_2)]}{[z,z_2]}$$
is again analytic in $\Omega_1$ and $|f_3(z)|\lesssim\eps$. Moreover, $f_3(z_j)=\hdel^1_j$ for $j=3,\dots,n$. Inductively, consider the function
$$f_{k+2}(z)=\frac{[f_{k+1}(z),f_{k+1}(z_{k+1})]}{[z,z_{k+1}]},$$
which is analytic in $\Omega_1$, bounded by a constant comparable to $\eps$ and $f_{k+2}(z_j)=\hdel^{k}_j$ for $j=k+2,\dots,n$. Then we may apply Lemma 
\ref{lemma1} to conclude that
$$|\hdel^k_{k+2}-\hdel^k_j|\lesssim\eps\rho(z_{k+2},z_j)\,\,\text{ for }j=k+3,\dots,n.$$

\medskip

\noindent Hence, choosing $\eps$ small enough, (\ref{ecc2}) holds for $n-1$ points. Clearly it also holds for any permutation of them. We can consequently 
apply Theorem \ref{main_th} to assert that there exists $h\in\U$ such that $h(z^{(j)}_m)=\hdel^{(j)}_m$ for $j=2,\dots,n$ and $m\in\N$.

\bigskip

\noindent Consider now the second interpolation problem defined by the values

\begin{equation*}
 \tdel^{(j)}_m=2\cdot\frac{w^{(j)}_m-w^{(1)}_{i(m)}}{B_1(z^{(j)}_m)\cdot E_1(z^{(j)}_m)}.
\end{equation*}

\noindent Observe that (\ref{cc1}) and (\ref{prop1}) imply that $|\tdel^{(j)}_m|\lesssim\eps$. As in the previous problem, we can consider two different cases. 
If $\beta(z^{(j)}_m,z^{(i)}_l)\geq\eta$, arguing as above we can see that for $\eps$ small enough condition (\ref{ecc}) holds. Hence, we may focus again on the 
domain $\Omega_1$. Observe the following tables of hyperbolic difference quotients

\begin{multicols}{2}
 \begin{displaymath}
 \begin{array}{|c|c|cccc|}
  \hline
  \vspace{0cm} & & & & &\\
  z_1 & \Delta^0_1 & & & &\\
  \vspace{0cm} & & & & &\\
  z_2 & \Delta^0_2 & \Delta^1_2 & & & \\
  \vspace{0cm} & & & & & \\
  z_3 & \Delta^0_3 & \Delta^1_3 & \Delta^2_3 & & \\
  \vspace{0cm} & & & & & \\
  \vdots & \vdots & \vdots & \vdots & \ddots & \\
  \vspace{0cm} & & & & & \\
  z_n & \Delta^0_n & \Delta^1_n & \Delta^2_n & \cdots & \Delta^{n-1}_n \\
  \vspace{0cm} & & & & & \\
  \hline
 \end{array}
\end{displaymath} 

\begin{displaymath}
 \begin{array}{|c|c|ccc|}
  \hline
  \vspace{0cm} & & & & \\
  z_2 & \tdel^0_2 & & & \\
  \vspace{0cm} & & & & \\
  z_3 & \tdel^0_3 & \tdel^1_3 & &  \\
  \vspace{0cm} & & & & \\
  \vdots & \vdots & \vdots & \ddots & \\
  \vspace{0cm} & & & & \\
  z_n & \tdel^0_n & \tdel^1_n & \cdots & \tdel^{n-2}_n \\
  \vspace{0cm} & & & & \\
  \hline
 \end{array}
\end{displaymath} 
\end{multicols}

\noindent where
$$\tdel^0_j=2\cdot\frac{\Delta^0_1-\Delta^0_j}{B_1(z_1)\cdot E_1(z_1)}\,\,\text{ for }j=2,\dots,n$$
and
$$\tdel^k_j=\frac{[\tdel^{k-1}_j,\tdel^{k-1}_{k+1}]}{[z_j,z_{k+1}]}\,\,\text{ for }k=1,\dots,n-2\,\text{ and }j=k+2,\dots,n,$$

\medskip

\noindent Recall that by hypothesis $|\Delta^k_j|\leq\eps$ for $k=1,\dots,n-1$ and $j=k+1,\dots,n$. This fact will imply that there exists $C>0$ depending 
only on $z_1,\dots,z_n$ and $n\in\N$ such that $|\tdel^k_j|\leq C\eps$ for $k=1,\dots,n-2\,\text{ and }j=k+2,\dots,n$. To this end, Theorem \ref{NP_BRW} 
applied to the triangle on the left says that there exists $g\in\U$ such that $g(z_j)=\Delta^0_j$ for $j=1,\dots,n$. Arguing as in the proof of Remark 
\ref{remark}, we can take $g_0(z)=0$ and the formula (\ref{formula_NP}) generates a solution $g\in\U$ such that $g(z_j)=w_j$ for $j=1,\dots,n$ and such that 
$|\Delta^1g(z;z_1)|\leq C\eps$ for $z\in\D$, where $C>0$ is a constant depending on $n$.

\medskip

\noindent Let $B_1^*(z)=B_1(z)/[z,z_1]$. Then there exists $C=C(\eta)>0$ such that $|B_1^*(z)|\geq C$ when $z\in \Omega_1$. Besides, observe that
$$\tdel^0_j=2\frac{\Delta^0_1-\Delta^0_j}{B_1(z_j)E_1(z_j)}=2\frac{1-\overline{\Delta^0_1}\Delta^0_j}{B_1^*(z_j)E_1(z_j)}\cdot\Delta^1_j,$$
for $j=2,\dots,n$. Then we can define the function
\begin{equation*}
 \widetilde{g}(z)=2\frac{1-\overline{\Delta^0_1}g(z)}{B_1^*(z)E_1(z)}\Delta^1g(z;z_1),
\end{equation*}
which is analytic on $\Omega_1$. Since $|E_1(z)|\simeq|1-\overline{\Delta^0_1}g(z)|$ when $z\in \Omega_1$, there exists $C>0$ such that 
$|\widetilde{g}(z)|\leq C\eps$ on $\Omega_1$. Furthermore by (\ref{estab}), and choosing $\eps$ small enough, the function $\widetilde{g}$ is in $\U$ and interpolates the values $\tdel^0_j$ in the nodes $z_j$ for $j=2,\dots,n$. Now, an argument, based on Lemma \ref{lemma1}, 
similar to the one used in the previous interpolation problem with the functions $f_k$, shows that there exists a constant $\widetilde{C}>0$ depending only on the 
sequence and $n$ such that 
$$|\tdel^k_j|=|\Delta^k\widetilde{g}(z_j;z_2,\dots,z_{k+1})|\leq\widetilde{C}\eps$$
for $k=1,\dots,n-2$ and $j=k+2,\dots,n$. 

\medskip

\noindent Consequently, for $\eps$ small enough the values $\{\tdel^{(j)}_m\}$ satisfy condition (\ref{ecc2}). And it is clear for any permutation 
of the points $\{z_2,\dots,z_n\}$. Applying now Theorem \ref{main_th} to $Z^{(2)}\cup\dots\cup Z^{(n)}$, we get that there exists a function $\widetilde{h}\in\U$ 
such that $\widetilde{h}(z^{(j)}_m)=\tdel^{(j)}_m$ for $j=2,\dots,n$ and $m\in\N$.

\medskip 

\noindent Finally the function $\widetilde{f}=\dfrac{1}{2}(h+\widetilde{h})$ is in $\U$ and $\widetilde{f}(z^{(j)}_m)=\widetilde{w}^{(j)}_m$ for $j=2,\dots,n$ 
and $m\in\N$. Hence (\ref{solution}) is the solution of the problem for $n$ separated sequences, and Theorem \ref{main_th} is proved.

\bigskip

\section{Sampling sequences for $\U$.}

\noindent In \cite{BN} the authors defined the concept of sampling sequence for the Bloch space, and gave a characterization of it. Recall that an analytic 
function $f$ defined on the unit disc is in the Bloch space $\B$ if there exists $C>0$ such that
$$\|f\|_{\B}=\sup_{z,w\in\D}\frac{|f(z)-f(w)|}{\beta(z,w)}\leq C.$$
The quantity $\|f\|_{\B}$ defines a semi-norm, and so $\B$ becomes a Banach space with the norm $\|f\|=|f(0)|+\|f\|_{\B}$. One may easily see that
$$\|f\|_{\B}=\sup_{z\in\D}(1-|z|^2)|f'(z)|.$$

\noindent Hence, a sequence $Z=\{z_n\}$ is said to be sampling for $\B$ if there exists $0<C<1$ such that for every $f\in\B$ one has that
$$\sup_{n\neq m}\frac{|f(z_n)-f(z_m)|}{\beta(z_n,z_m)}\geq C\|f\|_{\B}.$$
In the cited paper it was proved that $Z$ is a sampling sequence if and only if $Z$ is $R$-dense in $\D$, that is, there exists $R>0$ such that for any 
hyperbolic disc $D_h$ of radius greater than $R$ one has that $Z\cap D_h\neq\emptyset$. In \cite{Seip} one can find a deeper description of this problem.

\medskip

\noindent Focusing on the similarities between the space $\B$ and the class of functions $\U$, we see that in both of them the functions satisfy a Lipschitz type 
inequality. In the case of Bloch functions, the inequality involves Euclidean distance in the image space, and the best Lipschitz constant depends on the norm of the function. In the case of analytic self-mappings of the disc the Lipschitz type inequality is the Schwarz-Pick Lemma, so the distance involved is 
the hyperbolic metric, and the Lipschitz constant there is 1. Therefore, a sampling problem for $\U$ may be understood as the hyperbolic version of a sampling 
problem for $\B$. Nevertheless, $\U$ is not a Banach space, so we cannot use a norm here for giving a definition.

\medskip

\noindent Let $f\in\U$. Observe that for any $z,w\in\D$ 
$$\frac{\beta(f(z),f(w))}{\beta(z,w)}\leq\sup_{z\in\gamma}|f^h(z)|,$$
where $\gamma$ is the geodesic joining $z$ and $w$ and the hyperbolic derivative $f^h$ was defined in (\ref{hyp_der}). So then we can define the quantity
$$N(f)=\sup_{z,w\in\D}\frac{\beta(f(z),f(w))}{\beta(z,w)}=\sup_{z\in\D}|f^h(z)|\leq1.$$
The inequality is a direct consequence of the Schwarz-Pick Lemma. Now we can give a definition of a sampling sequence for $\U$ analogous to the one in the case 
of $\B$.

\begin{defi}
 A sequence $Z=\{z_n\}$ is a sampling sequence for $\U$ if there exists a constant $C\in(0,1)$ such that 
 $$\sup_{n\neq m}\frac{\beta(f(z_n),f(z_m))}{\beta(z_n,z_m)}\geq C\cdot N(f)$$
 for any $f\in\U$.
\end{defi}

\medskip

\noindent As in the case of $\B$, this definition is conformally invariant, in the sense that if $Z$ is a sampling sequence for $\U$ and $\tau$ is an 
automorphism of the disc then the sequence $\tau(Z)=\{\tau(z_j)\}$ is also a sampling sequence with the same sampling constant $C$. 

\medskip

\noindent We have proved the following characterization of sampling sequences for $\U$.

\begin{thm}
 A sequence $Z$ of pairwise different points in $\D$ is a sampling sequence for $\U$ if and only if there exists $R>0$ such that the sequence is $R$-dense in $\D$. 
\end{thm}

\medskip

\noindent The characterization coincides with the one given for $\B$. Nevertheless the proof of the sufficiency is different from the one  given in \cite{BN}. 
Before proving the result, recall that the hyperbolic length of a curve $\gamma\in\D$ is
$$\ell_h(\gamma)=\int_{\gamma}\frac{|dz|}{1-|z|^2}.$$

\begin{proof}
 By conformal invariance, it is enough to prove the condition for a hyperbolic disc centered at the origin. Suppose that $Z=\{z_n\}$ is a sampling sequence for $\U$. Then for $f(z)=z/2$ there exists $C>0$ such that
 $$\sup_{n\neq m}\frac{\beta(f(z_n),f(z_m))}{\beta(z_n,z_m)}\geq C/2.$$
 Hence, for a fixed $\eta>0$, we have that
 $$\sup\left\{\frac{\beta(z/2,w/2)}{\beta(z,w)}\,:\,|z|>r,|w|>r,\,\beta(z,w)\geq\eta\right\}\xrightarrow[r\rightarrow1]\,0.$$
 
 \noindent Suppose now that $Z$ is $R$-dense in $\D$. We argue by contradiction. If $Z$ is not a sampling sequence for $\U$, then for every $k\in\N$ there exists 
 $f_k\in\U$ such that
 $$\frac{1}{N(f_k)}\sup_{n\neq m}\frac{\beta(f_k(z_n),f_k(z_m))}{\beta(z_n,z_m)}<\frac{1}{k}.$$
 Fix $k$ large enough. Without loss of generality, we may take $f=f_k$ such that $f(0)=0$ and $|f'(0)|\geq N(f)/2$. Since $Z$ is $R$-dense, we can choose a 
 finite number $n(R)$ of points $z_j\in Z$ such that $R<\beta(0,z_j)<3R$, and $\beta(z_i,z_j)\geq R$. Assuming that $\arg{z_j}<\arg{z_{j+1}}$, let $\gamma$ be the 
 closed curve formed by the geodesics joining every $z_j$ with $z_{j+1}$ and $z_{n(R)}$ with $z_1$. Then, on the one hand
 $$\ell(f(\gamma))\leq\sum_{j}\beta(f(z_j),f(z_{j+1}))\leq\frac{N(f)}{k}\sum_{j}\beta(z_j,z_{j+1})\leq \frac{N(f)}{k}C_1(R),$$
 and on the other hand
 $$\ell(f(\gamma))=\int_\gamma\frac{|f'(z)|}{1-|f(z)|^2}|dz|\geq\int_\gamma|z|\left|\frac{f'(z)}{z}\right||dz|\geq C_2(R)|f'(0)|.$$
 Hence, there exists $C_3(R)>0$ such that $k\leq C_3(R)$. This finishes the proof.
\end{proof}

\bibliographystyle{plain}
\bibliography{monreal_papadimitrakis}

\end{document}